\documentclass[final,leqno,letterpaper]{etna}

\setbibdata{1}{xx}{xx}{2026} 

\usepackage{amssymb,amsmath,cite,booktabs,comment,graphicx}

\hypersetup{%
    pdftitle={A coercive space-time variational approach to fractional diffusion problems},
    pdfauthor={Herbert Egger, Marvin Fritz, Barbara Wohlmuth},
    pdfkeywords={fractional diffusion, fractional Sobolev spaces, coercivity, Galerkin discretization, space--time methods, a priori error analysis, fast convolution}
    }

\title{A coercive space-time variational approach to\\ fractional diffusion problems}

\author{Herbert Egger\footnotemark[1]
        \and Marvin Fritz\footnotemark[2] \and Barbara Wohlmuth\footnotemark[3]}

\shorttitle{A SPACE-TIME VARIATIONAL APPROACH TO FRACTIONAL DIFFUSION} 
\shortauthor{H.~EGGER, M.~FRITZ and B.~WOHLMUTH}

\newtheorem{remark}[theorem]{Remark}

\newtheorem{problem}[theorem]{Problem}

\begin{document}

\maketitle

\renewcommand{\thefootnote}{\fnsymbol{footnote}}

\footnotetext[2]{Institute of Numerical Mathematics, Johannes Kepler University Linz, and Johann Radon Institute for Computational and Applied Mathematics, 
Altenberger Str. 69, 4040 Linz, Austria.}
\footnotetext[3]{Faculty of Mathematics, University of Vienna,   Oskar-Morgenstern-Platz 1, 1090 Vienna, Austria.}
\footnotetext[4]{TUM School of Computation, Information and Technology, Technical University of Munich,
  Boltzmannstr. 3, 85748 Garching, Germany.}

\vspace{.2cm}

\begin{center}
\begin{minipage}{\textwidth} 
\em \footnotesize
This paper is dedicated to Olaf Steinbach on the occasion of his 60th
birthday. His pioneering work on polynomial approximation and on space--time variational formulations for evolutionary problems, in particular on coercive formulations based on modified Hilbert transforms, has provided important inspiration for the present work. 
\end{minipage}
\end{center}

\vspace{.2cm}

\begin{abstract}
We consider a fractional diffusion problem with temporal nonlocality acting on the diffusive flux. A coercive space--time variational formulation in Bochner-valued fractional Sobolev spaces is derived and the existence, uniqueness, and regularity of solutions are established. We further develop a conforming tensor-product Galerkin discretization and prove quasi-optimal error estimates in the anisotropic energy norm and improved convergence rates in weaker norms using duality arguments. In contrast to some space-time formulations for classical diffusion, the method preserves the causal structure of the evolution problem and leads to a time-stepping procedure with memory terms. On uniform time grids, the discrete history operator has a lower-triangular Toeplitz structure which enables an efficient implementation using fast recursive convolution techniques. 
\end{abstract}




\newcommand{\dt}{\partial_t}
\newcommand{\RR}{\mathbb{R}}
\newcommand{\CC}{\mathbb{C}}
\newcommand{\NN}{\mathbb{N}}
\newcommand{\VV}{\mathbb{V}}
\newcommand{\Om}{\Omega}
\newcommand{\I}{(0,T)}
\newcommand{\cA}{\mathcal{A}}
\newcommand{\cB}{\mathcal{B}}
\newcommand{\ip}[2]{\left(#1,#2\right)}
\newcommand{\norm}[1]{\left\|#1\right\|}
\newcommand{\seminorm}[1]{\left|#1\right|}
\newcommand{\Q}{Q}
\newcommand{\cF}{\mathcal{F}}

\newcommand{\IOP}{I_{0+}}
\newcommand{\ITM}{I_{T-}}
\newcommand{\DOP}{D_{0+}}
\newcommand{\DTM}{D_{T-}}
\newcommand{\ddt}{\tfrac{d}{dt}}
\newcommand{\F}{\mathcal{F}}
\newcommand{\A}{\mathcal{A}}
\newcommand{\uht}{u_{h\tau}}
\newcommand{\vht}{v_{h\tau}}
\newcommand{\tuht}{\tilde u_{h\tau}}
\newcommand{\tvht}{\tilde v_{h\tau}}

\section{Introduction}
\label{sec:intro}

Fractional differential equations arise naturally as effective models for anomalous transport in heterogeneous and multiscale media, emerging from microscopic random-walk dynamics or from upscaling procedures; see, e.g., \cite{LogvinovaNeel2004,CosenzaGiotGiraudHedan2021}. In contrast to classical diffusion
equations, the temporal nonlocality of fractional models globally couples the solution at different times and therefore leads to both analytical and computational
challenges. We refer to
\cite{Podlubny1999,Diethelm2010} for introduction to the field of fractional differential equations. 

As a model problem for this work, we consider the time-fractional diffusion equation
\begin{align}\label{eq:model-strong}
\dt u - \Delta\,\dt^{1-\alpha} u
    = f, \qquad 0 < \alpha < 1,
\end{align}
where $\dt^{1-\alpha}$ is the left-sided Riemann-Liouville fractional derivative; see below for a precise definition.
Equations of this type naturally arise, for example, in the coarse graining of multiscale models and stochastic evolution problems; see e.g. \cite{AngstmannHenry2020,MetzlerKlafter2000}. As outlined in these works, the fractional time-derivative naturally appears in the diffusive flux or transport terms. 
Related models arising in applications include the time-fractional Fisher--KPP models for tumor growth
\cite{fritz2026time} or time-fractional Fokker--Planck equations
\cite{fritz2024analysis}.

A standard approach to the analysis and numerical treatment of the fractional differential equations \eqref{eq:model-strong} is to apply the fractional integral $I^{1-\alpha}$ to the system. This formally leads to the more conventional sub-diffusion equation
\begin{align}\label{eq:caputo}
    \dt^\alpha u -\Delta u=g, \qquad 0 < \alpha < 1,
\end{align}
with $g=I^{1-\alpha}f$ introduced for abbreviation. In this context, $\dt^\alpha$ is usually understood as the Caputo fractional derivative. 
Existence and uniqueness of weak solutions, as well as regularity and stability estimates, have been studied extensively; see, e.g.,
\cite{LiXie2019,MuAhmadHuang2017,RahbyYang2025}.
A variety of numerical approaches to fractional diffusion have been developed,
including finite difference schemes such as the L1 and
Gr\"unwald--Letnikov schemes, finite element methods, and spectral methods; see \cite{JinLazarovZhou2013,LinXu2007,StynesORiordanGracia2017} for examples. A comprehensive overview about different approaches is provided in \cite{JinLazarovZhou2019}.
A somewhat different space-time variational approach to fractional diffusion has been  introduced by Li and Xu~\cite{LiXu2009}. They show that the differential operator arising in \eqref{eq:caputo} is in fact continuous and coercive on an appropriate sub-space of 
\begin{align*}
H^{\alpha/2}(0,T;L^2(\Omega)) \cap L^2(0,T;H^1(\Omega)). \end{align*}
This allows them to establish well-posedness of the weak form of this problem and the quasi-optimality of Galerkin approximations using results of elliptic theory. 
An alternative space-time framework based on inf-sup stability has been developed in~\cite{DuanEtAl2018}.

A substantial body of work has also been devoted to the investigation of space--time variational formulations for the classical diffusion equation 
\begin{equation}\label{eq:heat}
    \dt u-\Delta u=g,
\end{equation}
which can formally be interpreted as the limit of \eqref{eq:caputo} as
$\alpha\to1$. Let us note that the coercivity argument of
Li and Xu~\cite{LiXu2009} degenerates in this limit, since the corresponding stability constant tends to zero.
Coercive space--time formulations for
\eqref{eq:heat} can nevertheless be obtained by a suitable modification of
the test functions. 
In particular, Steinbach and collaborators apply
modified Hilbert transforms on the test space to construct coercive space--time formulations; see, e.g.,
\cite{Steinbach2015,SteinbachYang2019,SteinbachZank2020}.
Let us note that these integral transformations in the time variable are closely related to fractional integral operators and typically lead to space-time formulations with a global coupling in time. 
In particular, the locality and causality of the evolution problem \eqref{eq:heat} is usually lost during the discretization process. 
The efficient realization of such schemes has been addressed, for instance, in \cite{SchwabStevenson2009,HarbrechtSchwabZank2026}.
Other approaches leading to quasi-optimal Galerkin approximations can be found in \cite{Andreev2013,TantardiniVeeser2016}; see 
\cite{LangerSteinbach2019} for a collection of further results. 

\medskip
\noindent
\textbf{Main contributions.}
In this work, we extend the space--time variational approach of Li and Xu~\cite{LiXu2009} for equation~\eqref{eq:caputo}
to the time-fractional diffusion equation \eqref{eq:model-strong} in its natural form. Although related
in spirit, the resulting formulation leads to a different functional analytic setting, in which the weak solution is sought in an appropriate sub-space of 
\begin{align*}
    H^1(0,T;L^2(\Omega))
    \cap
    H^{1-\alpha/2}(0,T;H_0^1(\Omega)). 
\end{align*}
This allows us to incorporate initial conditions in a strong form, but requires additional arguments from fractional calculus. We establish continuity and coercivity of the bilinear form in the corresponding variational formulation and derive regularity results for the solution. Based on this framework, we develop a conforming tensor-product Galerkin discretization and prove quasi-optimal
error estimates in the natural space--time energy norm, as well as improved $L^2$-estimates by a duality argument. 
A distinctive feature of the proposed discretization is that it preserves the causal structure of the underlying fractional differential equation. In particular, it can be realized as a time-stepping scheme for an evolution problem with memory. On uniform time grids, the algebraic systems obtained after discretization have a lower-triangular Toeplitz structure, enabling fast history evaluation and efficient solution strategies.

\medskip
\noindent
\textbf{Outline.}
Section~\ref{sec:prelim} introduces the notation and the auxiliary results
from fractional calculus used throughout the paper.
In Section~\ref{sec:anal}, we derive the space--time variational formulation,
prove continuity and coercivity, and establish well-posedness and regularity
results for the solution.
Section~\ref{sec:disc} introduces the lowest-order conforming space--time
Galerkin discretization and derives a priori error estimates in
the energy norm, together with improved $L^2$-estimates based on a duality argument.
In Section~\ref{sec:fast}, we analyze the algebraic structure of the resulting
linear systems and exploit their lower-triangular structure for efficient
solution. On uniform time grids, the temporal history operator has an exact
lower-triangular Toeplitz structure, which enables fast convolution
techniques.
Section~\ref{sec:num} presents numerical experiments illustrating the
theoretical convergence rates and the computational performance of the
method.
Finally, Section~\ref{sec:final} discusses limitations of the present
analysis and possible extensions.

\section{Notation and fractional calculus}
\label{sec:prelim}

We start with establishing our notation and then recall some basic facts about fractional calculus. A comprehensive introduction to the field can be found in~\cite{Diethelm2010,SamkoKilbasMarichev1993}. We here only collect the most important results which are used in our analysis below. Our presentation closely follows \cite{LiXie2019}; for general notation, also see \cite{AdamsFournier2003}.

\subsection{Fractional Sobolev spaces}
We write $H^s(\mathbb R)$ for the fractional Sobolev spaces with norm $\|v\|_{H^s(\mathbb R)} = \|(1+\omega^2)^{s/2}\F(v)\|_{L^2(\mathbb R)}$ and $\F(v)(\omega) = (\omega)=(2\pi)^{-1/2}\int_{\mathbb R}e^{-i\omega t}v(t)\,dt$ denoting the Fourier transform of $v$. 
For an interval $(0,T)$, we define
\begin{align}
    H^s(0,T) = \{v|_{(0,T)}:v\in H^s(\mathbb R)\},
\end{align}
by restriction. These are Hilbert spaces when equipped with the quotient norm
\begin{align}
    \|v\|_{H^s(0,T)} = \inf_{\substack{\tilde v\in H^s(\mathbb R)\\\tilde v|_{(0,T)}=v}}
    \|\tilde v\|_{H^s(\mathbb R)}.
\end{align}
Let us recall that, for $s \ge 0$, the spaces $H^s(0,T)$ coincide, with equivalent norms, with the Sobolev-Slobodeckij spaces $W^{s,2}(0,T)$ and, for integer $s$, with the usual Sobolev spaces.
We will further use at different occasions that the embedding 
\begin{align}
H^s(0,T) \hookrightarrow C[0,T], \quad s > 1/2,
\end{align} 
is continuous and, in particular, the traces at $t=0$ and $t=T$ are well-defined for corresponding functions.
%
%
For details and the extension to multidimensional domains, see~\cite[Ch.~7]{AdamsFournier2003}.

\subsection{Fractional integrals}

Let $T>0$ and $\beta>0$ be given. For $v\in L^2(0,T)$, we denote the left- and right-sided Riemann--Liouville fractional integrals by
\begin{align}
    \IOP^\beta v(t)
    &=
    \frac{1}{\Gamma(\beta)}
    \int_0^t (t-s)^{\beta-1}v(s)\,ds,
    \label{eq:IOP}
    \\
    \ITM^\beta v(t)
    &=
    \frac{1}{\Gamma(\beta)}
    \int_t^T (s-t)^{\beta-1}v(s)\,ds.
    \label{eq:ITM}
\end{align}
We later write $I_+^\beta$ and $I_-^\beta$ for the corresponding fractional integrals on $\mathbb R$, which are obtained by replacing $0$ and
$T$ in \eqref{eq:IOP}--\eqref{eq:ITM} by $-\infty$ and $\infty$, respectively.
Let us recall that the left- and right-sided fractional integrals satisfy a semigroup property
\begin{align}     \label{eq:semi}
\IOP^\beta\IOP^\gamma v =\IOP^{\beta+\gamma}v
\qquad \text{and} \qquad 
\ITM^\beta\ITM^\gamma v = \ITM^{\beta+\gamma}v
\end{align}
for all $\beta,\gamma>0$ and $v \in L^2(0,T)$; see~e.g.~\cite[Theorem~2.2]{Diethelm2010}.
Moreover, the two operators are adjoint to each other, i.e.
\begin{align}
    (\IOP^\beta v,w)_{L^2(0,T)}
    =
    (v,\ITM^\beta w)_{L^2(0,T)}
    \qquad
    \forall\,v,w\in L^2(0,T).
    \label{eq:adj}
\end{align}
This formula allows to define the fractional integrals in a distributional sense and to extend their domain of definition to larger spaces; see e.g. \cite[Sec.~2]{LiXie2019} and below.
We shall also use the following coercivity estimate, which is derived in the spirit of \cite{ErvinRoop2006}; also see \cite{LiXu2009,LiXie2019}.

\begin{lemma}\label{lem:f1}
Let $0<\beta<1/2$. Then, for every $v\in L^2(0,T)$, there holds
\begin{align*} 
    (\IOP^\beta v, \ITM^\beta v)_{L^2(0,T)}
    \ge
    \cos(\pi\beta)
    \|\IOP^\beta v\|_{L^2(0,T)}^2.
\end{align*}
\end{lemma}
\begin{proof}
We denote by $\tilde v$ the extension of $v$ by zero to $\mathbb R$. Then 
\begin{align*}
(\IOP^{\beta}v,\ITM^{\beta}v)_{L^2(0,T)}
=(I_+^{\beta}\tilde v, 
      I_-^{\beta}\tilde v)_{L^2(\mathbb R)}.
\end{align*}
By the Fourier representation of fractional integrals~\cite[Ch.~7.1]{SamkoKilbasMarichev1993}, we know that
\begin{align*}
\F(I_{\pm}^\beta\tilde v)(\omega)
&= (\pm i\omega)^{-\beta}
    \F(\tilde v)(\omega).
\end{align*}
Together with the previous formula and some elementary transformations with complex fractions, this further leads to
\begin{align*}
    (\IOP^\beta v, \ITM^\beta v)_{L^2(0,T)}
    &=
    \int_{\mathbb R}
    (i\omega)^{-\beta} \,
    \overline{(-i\omega)^{-\beta}} \,
    |\F(\tilde v)(\omega)|^2\,d\omega
    \\
    &=
    \cos(\pi\beta)
    \int_{\mathbb R}
    |\omega|^{-2\beta}
    |\F(\tilde v)(\omega)|^2\,d\omega.
\end{align*}
Using the Fourier representation of fractional integrals once again and truncating the integration domain appropriately, we finally obtain
\begin{align*}
    (\IOP^\beta v, \ITM^\beta v)_{L^2(0,T)}
    =
    \cos(\pi\beta)
    \|I_+^{\beta}\tilde v\|_{L^2(\mathbb R)}^2
    \ge
    \cos(\pi\beta)
    \|\IOP^{\beta} v\|_{L^2(0,T)}^2.
\end{align*}
This already establishes the claimed coercivity estimate.
\end{proof}

\begin{remark}
We note that the norms  $\|\IOP^\beta v\|_{L^2(0,T)} \simeq \|\ITM^\beta v\|_{L^2(0,T)} \simeq \|v\|_{H^{-\beta}(0,T)}$ are equivalent for all $0 < \beta < 1/2$. This allows to extend the validity of the above formula to functions $v \in H^{-\beta}(0,T)$; we refer to \cite[Sec.~2]{LiXie2019}  for details. 
Using the adjointness relation and the semigroup property of fractional integrals, we can further see that
\begin{align*}
(\IOP^{2\beta} v, v)_{L^2(0,T)} 
= (\IOP^{\beta} v, \ITM^\beta v)_{L^2(0,T)} 
= (v, \ITM^{2\beta} v)_{L^2(0,T)}.
\end{align*}
We thus have established coercivity of the operators $\IOP^{2\beta}$ and $\ITM^{2\beta}$ for $0 < \beta < 1/2$. 
Further note that the coercivity constant $\cos (\pi \beta)$ degenerates as $\beta \to 1/2$. 
\end{remark}

\subsection{Fractional derivatives}
For $0<\beta<1$, the left- and right-sided
Riemann--Liouville fractional derivatives are defined by
\begin{align}
\DOP^\beta v = \dt\IOP^{1-\beta} v 
\qquad \text{and} \qquad 
\DTM^\beta v = -\dt\ITM^{1-\beta} v,
\end{align}
whenever the terms on the right hand sides make sense. 
The following results are well-known. 
\begin{lemma}\label{lem:f2}
Let $1/2<\beta<1$. Then the following assertions hold true.

\noindent
(i) If $v \in H^\beta(0,T)$ with $v(0)=0$, then
\begin{align}
    \DOP^\beta v
    &=
    \dt \IOP^{1-\beta}v
    =
    \IOP^{1-\beta} \dt v.
    \label{eq:f2a}
\end{align}
The last equality is understood in the sense of
distributions. \\[0.5ex]
(ii) The assumptions in (i) are satisfied if, and only if, $v\in L^2(0,T)$ and $\DOP^\beta v\in L^2(0,T)$. 
Similar to before, $\DOP^{\beta}$ here is understood in the sense of distributions.\\[0.5ex]
(iii) Under the conditions of point (i), one has 
\begin{align}
    c_{T,\beta}\|v\|_{H^\beta(0,T)}
    \le
    \|D_{0+}^\beta v\|_{L^2(0,T)}
    \le
    C_{T,\beta}\|v\|_{H^\beta(0,T)}.
    \label{eq:f2}
\end{align}
and the norm equivalence constants $c_{T,\beta}$, $C_{T,\beta}$ are independent of $v$.\\[0.5ex]
(iv)  Analogous statements hold for $\DTM^\beta$ with the endpoint condition at $t=T$.
\end{lemma}
\begin{proof}
For smooth functions $v \in H^1(0,T)$ with $v(0)=0$, the first assertion follows via integration-by-parts; it extends to $H^\beta(0,T)$ by density of $H^1(0,T)$ in $H^\beta(0,T)$. 
The remaining assertions can be found in 
\cite[Lemma~2.6]{LiXie2019}.
\end{proof}

\begin{remark} \label{rem:f2}
For later reference, we introduce the fractional Sobolev space 
\begin{align} 
H_{0+}^\beta(0,T)=\overline{C_c^\infty((0,T])}^{\|\cdot\|_{H^\beta(0,T)}}.
\end{align}
Then for $1/2<\beta<1$, we know from the previous lemma that 
\begin{align*}
H^\beta_{0+}(0,T)
&= \{v\in H^\beta(0,T):v(0)=0\} \\
&\hphantom{:}= \{v\in L^2(0,T):  \DOP^\beta v\in L^2(0,T)\},
\end{align*}
and the natural norms induced by the different definitions are equivalent. 
The norm equivalences \eqref{eq:f2} actually remain valid for all $v \in H_{0+}^\beta(0,T)$ and $0 < \beta < 1$; see~\cite{LiXu2009}.
Analogous results again also hold for the right-sided derivative and the endpoint $t=T$.
\end{remark}

\subsection{Further notation}
We use standard notation and write
$L^p(0,T;X)$ for the Bochner spaces of function $v : [0,T] \to X$. For brevity, we also write $L^p(X)$
if the time interval is clear from the context.
For $s\ge0$, we denote by $H^s(0,T;X)$ the corresponding
fractional Sobolev spaces in time and write $H^s(X)$ for abbreviation. 
For a domain $\Omega \subset \RR^d$, we denote by $L^p(\Omega)$ and $H^s(\Omega)$ the respective Lebesgue and Sobolev spaces. We again abbreviate with $L^p$ and $H^s$, if the domain of integration is clear. Let us recall that $L^p(\Omega \times (0,T)) = L^p(0,T;L^p(\Omega))$ which follows by Fubini's theorem.
For ease of notation, we sometimes write $a \lesssim b$ if $a \le C b$ with constant~$C$, whose meaning will be clear from the context, and $a \simeq b$ if $a \lesssim b$ and $b \lesssim a$.

\section{Space-Time Variational Formulation}
\label{sec:anal}
We can now formally introduce the model problem to be considered in the rest of the manuscript. 
Let $\Omega \subset \RR^d$, $d \ge 1$, be a bounded Lipschitz domain and let $T>0$. We consider the fractional diffusion equation
\begin{alignat}{2}
\dt u - \Delta \DOP^{1-\alpha} u &= f
    \qquad && \text{in } \Omega \times (0,T), \label{eq:sys1}\\
\intertext{for some $0 < \alpha < 1$ with homogeneous boundary and initial conditions, i.e.,}
u &= 0
    \qquad && \text{on } \partial\Omega \times (0,T), \label{eq:sys2}\\
u(0) &= 0
    \qquad && \text{in } \Omega. \label{eq:sys3}
\end{alignat}
Let us note that the extension to more general elliptic operators and other boundary conditions is straightforward. 
The incorporation of inhomogeneous initial conditions, on the other hand, would introduce singularities at $t=0$ which have to be treated with care; see e.g. \cite{LiXie2019}.
%

\subsection{Weak formulation}
Let $u$ be a sufficiently smooth solution of the initial-boundary value problem~\eqref{eq:sys1}--\eqref{eq:sys3}. 
We may then multiply the differential equation by a smooth test function
$\dt v$ and integrate over $Q=\Omega\times(0,T)$. 
This yields
\begin{align*}
(f,\dt v)_{L^2(Q)}
&=
(\dt u,\dt v)_{L^2(Q)}
+
(\nabla \DOP^{1-\alpha}u,\nabla \dt v)_{L^2(Q)}
=: (i)+(ii).
\end{align*}
We here integrated by parts with respect to the spatial variable $x$ and assumed that $v|_{\partial\Omega}=0$, such that the boundary terms vanish. 
We next interchange spatial and temporal derivatives and introduce $\psi=\nabla u$ and $\eta=\nabla v$ for abbreviation.
This allows us to rewrite the last term as 
\begin{align*}
(ii)
&=
(\DOP^{1-\alpha}\psi,\dt\eta)_{L^2(Q)}
=
(\IOP^\alpha\dt\psi,\dt\eta)_{L^2(Q)} \\
&=
(\IOP^{\alpha/2}\IOP^{\alpha/2}\dt\psi,
 \dt\eta)_{L^2(Q)}
= (\IOP^{\alpha/2}\dt\psi,
 \ITM^{\alpha/2}\dt\eta)_{L^2(Q)},
\end{align*}
where we used the identities \eqref{eq:f2a}, \eqref{eq:semi}, and \eqref{eq:adj} for the last three steps. 
We thus observe that all sufficiently smooth solutions $u$ of our model problem \eqref{eq:sys1}--\eqref{eq:sys3} satisfy the variational principle $\A(u,v)=\ell(v)$ for appropriate test functions $v$ 
with bilinear resp. linear forms
\begin{align}
\A(u,v)
&=
(\dt u,\dt v)_{L^2(Q)}
+
(\IOP^{\alpha/2}\dt \nabla u,
 \ITM^{\alpha/2}\dt \nabla v)_{L^2(Q)},
\label{eq:bilinear-form}
\\
\ell(v)
&=
(f,\dt v)_{L^2(Q)}.
\end{align}
We will show that all terms in these formulas are well-defined for functions $u,v$ in the space
\begin{align}
\VV
=H^1(0,T;L^2(\Omega)) \cap H_{0+}^{1-\alpha/2}(0,T;H^1_0(\Omega))
\end{align}
which thus is the natural energy space for variational treatment of problem~\eqref{eq:sys1}--\eqref{eq:sys3}.
This motivates the following weak formulation. 
\begin{problem} \label{prob:weak}
Find $u\in\VV$ such that
\begin{align} \label{eq:var}
\A(u,v)=\ell(v)
\qquad\forall v\in\VV.
\end{align}
Any solution of this problem will be called a weak solution to \eqref{eq:sys1}--\eqref{eq:sys3}.
\end{problem}
%

\subsection{Well-posedness}
The analysis of the space-time variational problem \eqref{eq:var} is surprisingly simple. The following auxiliary results provide the key ingredients.
\begin{lemma} \label{lem:aux}
Let $\Omega\subset\RR^d$, $d \ge 1$ be some bounded Lipschitz domain, $T>0$, and $0 < \alpha < 1$.
Then the following assertions hold true. \\[0.5ex]
(i) The expression
\begin{align} \label{eq:norm}
\|v\|_{\VV}
=
\left(
\|\dt v\|_{L^2(Q)}^2
+
\|\DOP^{1-\alpha/2}\nabla v\|_{L^2(Q)}^2
\right)^{1/2}
\end{align}
defines a norm on $\VV$ which is equivalent to the graph norm on $\VV$. Thus, $\VV$ is a Hilbert space when equipped with this norm. \\[0.5ex]
(ii) The mapping $\A:\VV\times\VV\to\RR$
is well-defined, bilinear, and continuous, i.e.,
\begin{align}
|\A(u,v)|
\leq
C_\A \|u\|_{\VV}\|v\|_{\VV}
\qquad\forall u,v\in\VV.
\end{align}

\noindent
(iii) Moreover, $\A(\cdot,\cdot)$ is coercive, i.e.,
\begin{align}
\A(u,u)
\geq
c_\A\|u\|_{\VV}^2 \qquad \forall u \in \VV.
\end{align}
Furthermore, the constants $C_\A,c_\A>0$ in these estimates only depend on $\alpha$ and $T$.
\end{lemma}

\begin{proof}
(i) Let $v \in \VV$. Then $v \in H_{0+}^{1-\alpha/2}(0,T;H_0^1(\Omega))$ and, hence, $v(0)=0$. Let us note that $\beta=1-\alpha/2>1/2$, so that the initial values are well-defined. 
As a consequence, the term $\|\dt v\|_{L^2(Q)}$ already defines a norm on $\VV$ and $\|v\|_{H^1(0,T;L^2(\Omega)} \simeq \|\dt v\|_{L^2(Q)}$.
By Lemma~\ref{lem:f2} and the Friedrichs inequality for $H_0^1(\Omega)$, we further conclude that 
$$
\|\DOP^{1-\alpha/2} \nabla v\|_{L^2(Q)} \simeq \|\nabla v\|_{H^{1-\alpha/2}(0,T;L^2(\Omega))} \simeq \|v\|_{H^{1-\alpha/2}(0,T;H^1(\Omega))}
$$ 
for all $v \in \VV$. This proves the equivalence of norms and also settles the last claim. \\[0.5ex]
(ii)
The first term in the definition of $\A(\cdot,\cdot)$ clearly satisfies the required continuity estimate.
For the second term, we note that
\begin{align*}
\IOP^{\alpha/2}\dt\nabla u
=
\DOP^{1-\alpha/2}\nabla u
\in L^2(Q)^d,
\end{align*}
which follows from \eqref{eq:f2a}.
For the test function, we use that elements of $\VV$ have well-defined temporal traces, and we define 
$\tilde v(t)=v(t)-v(T)$.
Then
\begin{align*}
\ITM^{\alpha/2}\dt\nabla v
=
\ITM^{\alpha/2}\dt\nabla\tilde v
=
-\dt\ITM^{\alpha/2}\nabla\tilde v
=
\DTM^{1-\alpha/2}\nabla\tilde v.
\end{align*}
In the two last steps, we again used the assertions of Lemma~\ref{lem:f2}. 
Set $\beta=1-\alpha/2$ for abbreviation. Then from the claims of this Lemma for $\ITM^\beta$, we further get
\begin{align*}
\|\DTM^\beta \nabla\tilde v\|_{L^2(Q)}
&\lesssim 
\|\nabla\tilde v\|_{H^\beta(0,T;L^2(\Omega))}
\\
&\le
\left(
\|\nabla v\|_{H^\beta(0,T;L^2(\Omega))}
+
\|\nabla v(T)\|_{L^2(\Omega)}
\right)
\\
&\lesssim
\|\nabla v\|_{H^\beta(0,T;L^2(\Omega))}
\lesssim 
\|\DOP^\beta\nabla v\|_{L^2(Q)}.
\end{align*}
Here we used the definition of $\tilde v$ and the triangle inequality in the second step, the continuous embedding of $H^\beta(0,T) \hookrightarrow C[0,T]$ in the third, and the norm equivalences of Lemma~\ref{lem:f2} in the first and last step. 
This shows that also the second term in the bilinear form $\A(\cdot,\cdot)$ is well-defined and continuous. 
The constants in the above estimates only depend on $T$ and $\alpha$. \\[0.5ex]
(iii)
Using Lemma~\ref{lem:f1} and $\IOP^{\alpha/2}\dt\nabla u
=\DOP^{1-\alpha/2}\nabla u$, see Lemma~\ref{lem:f2},
we obtain
\begin{align*}
\A(u,u)
&\geq
\|\dt u\|_{L^2(Q)}^2
+
\cos\big(\tfrac{\pi\alpha}{2}\big)
\|\IOP^{\alpha/2}\dt\nabla u\|_{L^2(Q)}^2
\geq
\cos\big(\tfrac{\pi\alpha}{2}\big)
\|u\|_{\VV}^2.
\end{align*}
This yields coercivity with the explicit constant 
$c_\A=\cos\left(\frac{\pi\alpha}{2}\right)$.
\end{proof}

\begin{remark}
Let us note that the coercivity constant $c_\A=\cos\big(\tfrac{\pi\alpha}{2}\big)$ derived in the previous proof degenerates as $\alpha\to 1$. This coincides with the fact that the investigated space-time bilinear form $\A(\cdot,\cdot)$ loses ellipticity in the standard diffusion case. Similar observations hold for other weak formulations; see \cite{LiXu2009} and \cite{SteinbachZank2020}.
\end{remark}

The previous results allow us to apply arguments developed for elliptic variational problems in order to prove the existence of a unique weak solution.
\begin{theorem} \label{thm:wp}
For every $f\in L^2(Q)$, the initial-boundary value problem~\ref{eq:sys1}--\eqref{eq:sys3} admits a unique weak solution $u\in H^1(0,T;L^2(\Omega)) \cap H_{0+}^{1-\alpha/2}(0,T;H_0^1(\Omega))$.
\end{theorem}

\begin{proof}
The claim follows by Lemma~\ref{lem:aux} and the Lax--Milgram lemma~\cite[Ch.~6.2]{Evans2010}.
\end{proof}

\subsection{Regularity}
\label{sec:reg}
For completeness of the presentation and later reference, we also derive some regularity results for solutions of the fractional diffusion equation \eqref{eq:sys1}. 
Related results can be found in \cite{McLean2011} and corresponding assertions concerning equation~\eqref{eq:caputo} can be found in \cite{LiXie2019}. 
We here proceed in a similar manner as for the heat equation; see \cite[Ch.~7.1.3]{Evans2010}.

Let $u$ be a weak solution to \eqref{eq:sys1}--\eqref{eq:sys3} with $f \in L^2(Q)$. 
By formally substituting $u$ into the strong form of the fractional differential equation, we see that 
\begin{align*}
-\Delta\DOP^{1-\alpha}u
=
f-\dt u
=:\tilde f,
\end{align*}
which is to be understood in the sense of distributions. 
From the natural regularity of $f$ and $\dt u$, we know that $\tilde f \in L^2(Q)$. 
By an application of the norm-equivalences stated in Lemma~\ref{lem:f2} and Remark~\ref{rem:f2}, we can thus deduce that $-\Delta u \in H^{1-\alpha}(0,T;L^2(\Omega))$.
On smooth or convex domains, elliptic regularity, see \cite[Ch.~6.3]{Evans2010}, then further implies that 
\begin{align*}
u\in H^1(0,T;L^2(\Omega)) \cap H^{1-\alpha}(0,T;H^2(\Omega)).
\end{align*}
This also amounts to the maximal regularity one can expect for data $f \in L^2(Q)$.

Higher temporal regularity can be obtained by formally differentiating the fractional diffusion equation~\eqref{eq:sys1} in time.
Under suitable temporal regularity, the function $u'=\dt u$ can then be seen to satisfy 
\begin{align*}
\dt u' - \Delta\DOP^{1-\alpha}u'
=
\dt f =: f'.
\end{align*}
Here we have used the first assertion of Lemma~\ref{lem:f2} to commute one temporal derivative in the fractional term $\dt \DOP^{1-\alpha} u = \DOP^{1-\alpha} \dt u$.
If $f\in H^1(0,T;L^2(\Omega))$ with $f(0)=0$, then 
\begin{align} \label{eq:der}
u'(0)=\dt u(0)
=
f(0)+\Delta\DOP^{1-\alpha}u(0)
=
0.
\end{align}
The function $u'$ thus again satisfies a homogeneous initial condition. From Theorem~\ref{thm:wp} applied to \eqref{eq:der} and the previous regularity and density arguments, we conclude that 
\begin{align*}
u'
\in
H^1(0,T;L^2(\Omega))
\cap
H^{1-\alpha}(0,T;H^2(\Omega))
\end{align*}
if the domain $\Omega$ is assumed to be convex; the assumption on the data could potentially be weakened somewhat.
In summary, we thus have established the following result.
\begin{theorem} \label{thm:reg} 
Let $\Omega$ be convex. Then for $f\in L^2(Q)$, the weak solution of the initial-boundary value problem \eqref{eq:sys1}--\eqref{eq:sys3} satisfies 
\begin{align}
u
\in
H^{1-\alpha}(0,T;H^2(\Omega)).
\end{align}
If $f \in H^1(0,T;L^2(\Omega))$ and $f(0)=0$, then 
\begin{align}
u
\in
H^2(0,T;L^2(\Omega)) 
\cap H^{2-\alpha/2}(0,T;H^1(\Omega)
\cap H^{2-\alpha}(0,T;H^2(\Omega)).
\end{align}
In both cases, the respective norms of the solution can be bounded by those of the data. 
\end{theorem}

Similar regularity results for equation~\eqref{eq:caputo}, including inhomogeneous and incompatible initial and boundary conditions, can be found in the paper of Li and Xie~\cite{LiXie2019}.

\section{Discretization and error analysis}
\label{sec:disc}
For discretization of \eqref{eq:sys1}--\eqref{eq:sys3}, we can employ a standard Galerkin approximation of Problem~\ref{prob:weak} in space and time.
We here consider a tensor-product discretization by continuous piecewise linear functions in both variables. The corresponding approximation space is thus given by 
\begin{align} \label{eq:Vht}
    \VV_{h\tau}
    =
    \left\{
        v\in\VV:
        v|_{[t^{n-1},t^n]}
        \in P^1([t^{n-1},t^n];V_h)
    \right\},
\end{align}
where $V_h=P_1(\mathcal T_h)\cap H_0^1(\Omega)$ 
denotes the standard finite element space associated with a quasi-uniform simplicial mesh $\mathcal T_h$, and
$t^n=n\tau$, $\tau=\frac{T}{N}$,
defines a uniform temporal discretization. 
Let us note that our analysis, in fact, also covers locally refined meshes in space and time, as long as certain interpolation error estimates hold true; see Remark~\ref{rem:adaptive} below.
In the remainder of this section, we consider the following discrete problem.
\begin{problem}\label{prob:disc}
Find $\uht\in\VV_{h\tau}$ such that
\begin{align} \label{eq:disc}
    \A(\uht,\vht)
    =
    \ell(\vht)
    \qquad
    \forall \vht\in\VV_{h\tau}.
\end{align}
\end{problem}
Since the approximation space $\VV_{h\tau}\subset\VV$ is finite-dimensional, existence and
uniqueness of the discrete solution $\uht$ follow immediately from the ellipticity and continuity of the bilinear form stated in Lemma~\ref{lem:aux} and using the Lax--Milgram lemma again.
%

\subsection{Error estimates in the energy-norm}
We start with recalling some elementary facts about polynomial approximation; see e.g. \cite{BrennerScott2008,Steinbach2008} for a comprehensive introduction. 
We write $\Pi_h:L^2(\Omega)\to V_h$
for the $L^2$-orthogonal projection in space, and recall the interpolation error estimates
\begin{alignat}{2}
\|\Pi_hu-u\|_{L^2(\Omega)}
&\le C h^s\|u\|_{H^s(\Omega)},  \qquad && 0\le s\le2,   \label{eq:int1}\\
\|\Pi_hu-u\|_{H^1(\Omega)}
&\le C h^{s-1}\|u\|_{H^s(\Omega)}, \qquad &&1\le s\le2. \label{eq:int2}
\end{alignat}
As usual, $h = \max_{T \in T_h} h_T$ and $h_T=\text{diam}(T)$ denote the global and local mesh size here.
We further write $I_\tau$ for the piecewise linear interpolation operator in time which is well-defined for continuous functions and satisfies
\begin{alignat}{2}
\|I_\tau u-u\|_{H^{1-\beta}(0,T)}
&\le C \tau^{r+\beta} \|u\|_{H^{1+r}(0,T)},
    \quad && -\beta\le r\le1,
    \
    0\le\beta<1/2. \label{eq:int3}
\end{alignat}
These estimates follow from standard polynomial approximation results and 
interpolation of Sobolev spaces; see e.g. 
\cite[Ch.~9.3 and Ch.~10.2]{Steinbach2008}.
An appropriate combination of these interpolation error estimates allows us to establish the following result.
\begin{lemma}\label{lem:approx} 
Let $0\le r\le1$ and $0\le s\le1$, and define $\tvht = \Pi_h I_\tau v$. Then 
\begin{align*}
\|v-\tvht\|_\VV
&\le C \tau^r 
\big(
    \|v\|_{H^{1+r}(0,T;L^2(\Omega))}
    +
    \|v\|_{H^{1-\alpha/2+r}(0,T;H^1(\Omega))}
\big)
\nonumber\\
&\qquad
+ C' h^s 
\big(
    \|b\|_{H^1(0,T;H^s(\Omega))}
    +
    \|b\|_{H^{1-\alpha/2}(0,T;H^{1+s}(\Omega))}
\big).
\end{align*}
The estimate applies to all functions $v$ whose norms are bounded appropriately.
\end{lemma}
\begin{proof}
By the triangle inequality, we can split 
\begin{align*}
    \|v- \tvht\|_\VV
    &\le
    \|v-I_\tau v\|_\VV
    +
    \|I_\tau (v-\Pi_h v)\|_\VV.
\end{align*}
Substitution into the definition of the norm $\|\cdot\|_\VV$ gives
\begin{align*}
\|v-I_\tau v\|_\VV
&\le \|\dt v-\dt I_\tau v\|_{L^2(Q)}
+ \|\DOP^{1-\alpha/2}(v-I_\tau\nabla v)\|_{L^2(Q)}
\\
&\le C \tau^r
\|v\|_{H^{1+r}(0,T;L^2(\Omega))}
+ C'\tau^r \|v\|_{H^{1-\alpha/2+r}(0,T;H^1(\Omega))},
\end{align*}
where we have used the temporal interpolation estimates and, for the
last term, the estimates from Lemma~\ref{lem:f2}.
For the spatial approximation error, we have
\begin{align*}
\|I_\tau(v-\Pi_h v)\|_\VV
&\le
\|\dt I_\tau(v-\Pi_h v)\|_{L^2(Q)}
+ \|\DOP^{1-\alpha/2} I_\tau\nabla(v-\Pi_h v)\|_{L^2(Q)}
\\
&\lesssim
\|v-\Pi_h v\|_{L^2(Q)}
+ \|v-\Pi_h v\|_{H^{1-\alpha/2}(0,T;H^1(\Omega))}.
\end{align*}
In the second step, we used the fact that
$\dt I_\tau=\Pi_\tau$ amounts to the temporal $L^2$-projection, which is contractive with respect
to the $L^2$-norm.
For the second term, we further employed the estimates from Lemma~\ref{lem:f2} together with the stability of the temporal interpolation operator in the $H^s$-norm, namely,
\begin{align*}
\|I_\tau v\|_{H^s(0,T)}
\le C'\|v\|_{H^s(0,T)},
    \quad 1/2 < s \le 1,
\end{align*}
which is a consequence from the embedding
$H^s(0,T)\hookrightarrow C[0,T]$ for $s>1/2$. It can be derived for arbitrary partitions  based on the definition of the Slobodeckij norms, a localized weighted version, support properties and suitable trace free Harding inequalities, see also \cite{Steinbach2008}.
By employing the spatial approximation estimates stated above, we further obtain 
\begin{align*}
\|I_\tau(v-\Pi_h v)\|_\VV
&\le
C h^s
\|v\|_{H^1(0,T;H^s(\Omega))}
\\
&\qquad + C' h^s
\|v\|_{H^{1-\alpha/2}(0,T;H^{1+s}(\Omega))}.
\end{align*}
A combination of the above estimates then yields the assertion of the lemma.
\end{proof}

\begin{remark} \label{rem:adaptive}
The above interpolation error estimates for $\Pi_h$ and $I_\tau$ actually hold for a rather general class of adaptively refined meshes in space and time; see \cite{BramblePasciakSteinbach2002,DieningStornTscherpel2021} and \cite[Ch.~10]{Steinbach2008} for details. 
The approximation error estimate of Lemma~\ref{lem:approx} and the error estimates of Theorem~\ref{thm:err1} and \ref{thm:err2} below, therefore, can be generalized quite easily to adaptive tensor-product meshes. Corresponding numerical results are presented in Section~\ref{subsec:numerics-graded}. 
Also the extension to higher-order polynomial approximations is rather straightforward. 
\end{remark}

We are now in the position to state and prove our first error estimate which provides optimal convergence rates under natural smoothness assumptions.
\begin{theorem}
\label{thm:err1}
(i) Let the solution $u$ of \eqref{eq:sys1}--\eqref{eq:sys3} satisfy the regularity conditions
\begin{align}
u &\in H^{1+r}(0,T;L^2(\Omega)) \cap H^{1-\alpha/2+r}(0,T;H^1(\Omega)) \qquad \text{and} \label{eq:ass1}\\
u & \in H^1(0,T;H^s(\Omega)) \cap H^{1-\alpha/2}(0,T;H^{1+s}(\Omega)) \label{eq:ass2}
\end{align}
for some $0\le r\le 1$ and $0\le s \le 1$.
Then there holds $\|u - \uht\|_\VV = O(\tau^r + h^s)$. \\[0.5ex]
(ii) If $u \in H^2(0,T;L^2(\Omega)) \cap H^{2-\alpha}(0,T;H^2(\Omega))$, then $\|u - \uht\|_\VV = O(\tau + h)$. 
\end{theorem}

\begin{proof}
By Lemma~\ref{lem:aux} and Céa's lemma, we immediately get
\begin{align*}
\|u-\uht\|_\VV \le \frac{C_\A}{c_\A} 
    \|u-\tuht\|_\VV
\end{align*}
for any $\tuht \in \VV$. We may choose $\tuht = \Pi_h  I_\tau u$ and employ Lemma~\ref{lem:approx} to obtain the first error estimate. 
If $u \in H^2(0,T;L^2(\Omega)) \cap H^{2-\alpha}(0,T;H^2(\Omega))$, then \eqref{eq:ass1}--\eqref{eq:ass2} are satisfied with $r=s=1$, which already yields the second claim.
\end{proof}

\begin{remark} \label{rem:energy}
The regularity assumptions of the previous theorem are natural. In particular, see Theorem~\ref{thm:reg}, we know that $u \in H^2(0,T;L^2(\Omega)) \cap H^{1-\alpha}(0,T;H^2(\Omega))$ is valid if $f\in H^1(0,T;L^2(\Omega))$ with $f(0)=0$ and the domain $\Omega$ is convex.
The convergence rate $O(h + \tau)$ obtained under these regularity assumptions is of optimal order. 
Further recall that $\|v\|_\VV \simeq \|v\|_{H^1(0,T;L^2(\Omega))} + \|v\|_{H^{1-\alpha/2}}(0,T;H^1(\Omega))$ for all $v \in \VV$. 
Thus the error estimates of the theorem  are also valid in standard Bochner-Sobolev norms. 
Since the embedding $H^{1-\alpha/2}(0,T)   \hookrightarrow C[0,T]$ is continuous, we obtain
\begin{align*}
    \|u-\uht\|_{L^\infty(0,T;H^1(\Omega))}
    \le
    C\|u-\uht\|_\VV,
\end{align*}
which immediately yields a corresponding error estimates in the $L^\infty(H^1)$-norm.
\end{remark}

\subsection{Improved estimates by a duality argument}

We next derive an estimate for the error in the $L^2(Q)$-norm by means of
a duality argument.
Let $g\in L^2(Q)$ be given and define
\begin{align}
    z(t)=\int_0^t w(s)\,ds,
\end{align}
where $w$ is the solution of the following dual problem:
\begin{alignat}{2}
-\dt w-\Delta\DTM^{1-\alpha}w
&=g
&&\qquad\text{in }\Omega\times(0,T),
\\
w&=0
&&\qquad\text{on }\partial\Omega\times(0,T),
\\
w(T)&=0
&&\qquad\text{in }\Omega.
\end{alignat}
If $\Omega$ is convex, then the solution of the dual problem satisfies
\[
    w\in
    H^1(0,T;L^2(\Omega))
    \cap
    H^{1-\alpha}(0,T;H^2(\Omega)),
\]
and its norm is bounded by the $L^2$-norm of $g$;
this follows with the same arguments as used in Section~\ref{sec:reg}.
Hence $z\in H^2(0,T;L^2(\Omega)) 
\cap H^{2-\alpha}(0,T;H^2(\Omega))$ and 
\begin{align*}
    \|z\|_{H^2(0,T;L^2(\Omega))}
    + \|z\|_{H^{2-\alpha}(0,T;H^2(\Omega))}
    \le C\|g\|_{L^2(Q)}.
\end{align*}
We now consider the dual solutions $w$ and $z$ for right hand side $g=u-\uht$. This yields
\begin{align*}
\|u-\uht\|_{L^2(Q)}^2
&= (g,u-\uht)_{L^2(Q)} \\
&= (-\dt w,u-\uht)_{L^2(Q)}
- (\Delta\DTM^{1-\alpha}w,u-\uht)_{L^2(Q)} \\
&= (\dt(u-\uht),w)_{L^2(Q)}
+ (\DOP^{1-\alpha/2}
    \nabla(u-\uht),
    \ITM^{\alpha/2}\nabla w)_{L^2(Q)}.
\end{align*}
In the last step, we used that $\DTM^{1-\alpha}w    =-\dt\ITM^{\alpha/2}\ITM^{\alpha/2}w$; see Lemma~\ref{lem:f2}, 
and we subsequently transferred the first two operators to the other side by integration by parts and adjointness; see Lemma~\ref{lem:f1}.
We now substitute $w=\dt z$ in this formula and define its approximation $\tilde z_{h\tau} = \Pi_h I_\tau z$.
Then
\begin{align*}
\|u-\uht\|_{L^2(Q)}^2
&=
(\dt(u-\uht),\dt z)_{L^2(Q)}
+ (\IOP^{\alpha/2}\dt\nabla(u-\uht),
    \ITM^{\alpha/2}\dt z)_{L^2(Q)}
\\
&=
\A(u-\uht,z) 
= \A(u-\uht,z-\tilde z_{h\tau}).
\end{align*}
The last identity follows from Galerkin orthogonality, which is implied by the weak form of the continuous and the discrete problem. By continuity of the bilinear form, we then get
\begin{align*}
\|u-\uht\|_{L^2(Q)}^2
&\le C_\A \|u-\uht\|_\VV \|z-\tilde z_{h\tau}\|_\VV
\\
&\le C (\tau+h)^2 \|u\|_{H^2(L^2)\cap H^{2-\alpha}(H^2)}
\|z\|_{H^2(L^2)\cap H^{2-\alpha}(H^2)}.
\end{align*}
Since the norm of $z$ can be bounded by the $L^2$-norm of the error $g=u-\uht$, the previous considerations  immediately lead to the following results. 
\begin{theorem}[$L^2$-error estimate]  \label{thm:err2} 
Let $u\in H^2(0,T;L^2(\Omega)) \cap
    H^{2-\alpha}(0,T;H^2(\Omega))$ and assume that $\Omega$ is convex. 
    Then $\|u - \uht\|_{L^2(Q)} =  O(\tau^2+h^2)$.
\end{theorem}

The proof of this result follows immediately from the preceding discussion. 
\begin{remark} \label{rem:duality}
We again define $\tuht = \Pi_h I_\tau u$ for approximation. 
Then, by the interpolation error estimates stated above, one can see that
\begin{align*}
\|u-\tuht\|_{L^2(Q)}
&\le \|u-\Pi_hu\|_{L^2(Q)} + \|\Pi_h(u-I_\tau u)\|_{L^2(Q)} \\
&\le C h^2 \|u\|_{L^2(0,T;H^2(\Omega))}
+ C' \tau^2 \|u\|_{H^2(0,T;L^2(\Omega))}.
\end{align*}
Thus, the error estimate from the preceding theorem yields the optimal achievable convergence rates, although it is somewhat suboptimal with respect to the required regularity of $u$. 
By the standard interpolation inequality 
$\|v\|_{L^\infty(0,T)} \le  C \|v\|_{L^2(0,T)}^{1/2}
    \|v\|_{H^1(0,T)}^{1/2}$ 
and the error estimates obtained before, we can further see that
\begin{align*}
\|u-\uht\|_{L^\infty(0,T;L^2(\Omega))}
&\lesssim
\|u-\uht\|_{L^2(Q)}^{1/2}
\|u-\uht\|_{H^1(0,T;L^2(\Omega))}^{1/2} 
\lesssim (\tau+h)^{3/2}.
\end{align*}
Here we assumed that the regularity hypotheses of the preceding lemma hold.
\end{remark}

\section{Efficient implementation}
\label{sec:fast}

The tensor-product structure of the space-time discretization can be exploited for an efficient numerical realization. 
A first key observation is that the variational formulation involves the temporal derivatives only. We therefore introduce
\begin{align*}
    p=\dt u,  \qquad q=\dt v.
\end{align*}
Since the functions in $\VV$ satisfy homogeneous initial condition, $u$ can be recovered from $p$ by 
\begin{align*}
    u(t)=\int_0^t p(s)\,ds.
\end{align*}
Thus, instead of continuous piecewise linear basis functions $\psi_j$ for $u$ and $v$, we can use discontinuous piecewise constant basis functions $\chi_j = \dt \psi_j$ for their temporal derivatives $p=\dt u$ and $q = \dt v$.
For the tensor-product discretization \eqref{eq:Vht}, this leads to a representation of the discrete bilinear form with the special structure
\begin{align}\label{eq:fast-tensor}
    \A(u_{h\tau},v_{h\tau})
    =
    \mathbf q^\top
    \left(
        D_\tau\otimes M_h
        +
        C_\tau\otimes K_h
    \right)
    \mathbf p,
\end{align}
where $M_h$ and $K_h$ denote the spatial mass and stiffness matrices,
respectively, and the vectors $\mathbf{p}$ and $\mathbf{q}$ represent the degrees of freedom of $p_{h\tau}=\dt \uht$ and $q_{h\tau}=\dt \vht$.  
The entries of the temporal matrices in \eqref{eq:fast-tensor} are then given by
\begin{align}
(D_\tau)_{ij} 
  &=(\chi_j,\chi_i)_{L^2(0,T)}
   = \tau_i\delta_{ij}, \label{eq:Wtau}  
\\
(C_\tau)_{ij}
  &= ( I_{0+}^{\alpha/2}\chi_j,  I_{T-}^{\alpha/2}\chi_i )_{L^2(0,T)}
   = ( I_{0+}^{\alpha}\chi_j, \chi_i )_{L^2(0,T)}, \label{eq:Ktau}
\end{align}
where $\chi_i$ denotes the indicator function for the $i$th time interval. For the second equation, we used the adjointness relation \eqref{eq:adj} and the semigroup property \eqref{eq:semi} for the fractional integrals.
Let us note that $D_\tau$ is diagonal and the matrix $C_\tau$ is lower triangular, since $\IOP^\alpha$ is a causal operator. 
Moreover, the entries of $C_\tau$ can be computed analytically for general time grids.

\begin{remark} 
With the above modifications, the numerical solution of the discrete system \eqref{eq:disc} can be reduced to the solution of a lower-triangular block system with sparse blocks. In particular, method \eqref{eq:disc} can be realized as a time-stepping scheme, which requires the solution of one elliptic finite element system in space for every time step.
Assuming that the spatial finite element systems can be solved in $N_x^\gamma$ operations, the space-time system can thus be solved in $O(N_x^\gamma N_t + N_x N_t^2)$ complexity using block-forward substitution. The second part stems from the recursive update of the residuals which can, however, be improved substantially using matrix compression techniques~\cite{DoelzEggerShashkov2021,ZengTurnerBurrage2018}.
On a uniform time grid, the system matrix has the particular form 
\begin{align} \label{eq:toeplitz}
    \tau I \otimes M_h
    + C_\tau\otimes K_h
\end{align}
and the entries of the fractional integral matrix $C_\tau$ are given by 
 \begin{align*} 
(C_\tau)_{ij} &=  
    \tau^{\alpha+1} \tfrac{(k+1)^{\alpha+1}
        -2k^{\alpha+1}
        +(k-1)_+^{\alpha+1}
}{\Gamma(\alpha+2)},
    \qquad
    k=i-j\ge0,
\end{align*}
with $(a)_+=\max(a,0)$. 
Hence \eqref{eq:toeplitz} is a lower-triangular block-Toeplitz matrix with sparse blocks. 
The complexity of the numerical solution can be reduced to $O(N_x^\gamma N_t + N_x N_t \log^2 N_t )$ using recursive fast Toeplitz multiplications for the residual computation~\cite{HairerLubichSchlichte1985}. 
We further note that Bini's approximate circular embedding algorithm~\cite{Bini1984} allows a parallel-in-time solution of the problem in $O( N_x^\gamma N_t + N_x N_t \log N_t )$ complexity; see \cite{LuPangSun2015} for related work.
\end{remark}

\section{Numerical results}
\label{sec:num}

We next illustrate the theoretical results of the method discussed in the previous sections by some numerical tests. 
In all our computations, we consider the model problem \eqref{eq:sys1}--\eqref{eq:sys3} on the unit square
$\Omega=(0,1)^2$ with $T=1$, and we use
$\alpha=0.6$ throughout.
For our computations, we consider the manufactured solution
\begin{align}\label{eq:numerical-exact-solution}
    u(x,y,t)=t^\beta\sin(\pi x)\sin(\pi y).
\end{align}
This is a solution of \eqref{eq:sys1}--\eqref{eq:sys3} with the corresponding right-hand side given by
\begin{equation}\label{eq:numerical-forcing-2d}
f(x,y,t)
=
\left[
\beta t^{\beta-1}
+
2\pi^2
\tfrac{\Gamma(\beta+1)}{\Gamma(\beta+\alpha)}
t^{\beta+\alpha-1}
\right]
\sin(\pi x)\sin(\pi y).
\end{equation}
The solution $u$ is smooth in space, while the parameter $\beta$ allows us to control its temporal regularity. For non-integer $\beta$, we have $t^\beta \in H^{1+r}(0,T)$ if, and only if, $r<\beta-\tfrac12$, and hence
\begin{align*}
    u \in H^{1+r}(0,T;H^2(\Omega))
    \qquad \text{if, and only if,} \quad  \beta > r + \tfrac12 \text{ for } \beta \not\in\NN.
\end{align*}
For $\beta \in \NN$ the solution $u$ is smooth in time as well. 
In our numerical tests, we first examine the smooth case $\beta=2$, and then investigate the cases $\beta=0.75$ and $\beta=1.25$ leading to reduced temporal regularity. 
We use spatial meshes obtained by subdividing a uniform  $n_x\times n_x$ Cartesian grid into triangles with mesh size $h\simeq n_x^{-1}$ and $N_x = (n_x-1)^2$ interior vertices.  
Unless stated otherwise, we used a uniform partition of the time interval with step size $\tau=T/N_t$. 
All experiments were run on a MacBook Air with M1 processor and $16\,\mathrm{GB}$ RAM. Spatial finite-element assembly was
performed with \texttt{Firedrake} \cite{rathgeber2016firedrake}. The temporal block algebra, sparse 
factorizations, and FFTs were implemented with \texttt{NumPy} and \texttt{SciPy}. %

\subsection{Convergence for smooth solutions}
\label{subsec:numerics-balanced}

We choose $\beta=2$ such that the manufactured solution \eqref{eq:numerical-exact-solution} is smooth in space and time. In particular, $u \in H^2(0,T;H^2(\Omega))$ such that the regularity assumptions of Theorems~\ref{thm:err1}
and~\ref{thm:err2} are satisfied. Hence we expect
\begin{align*}
\|u - \uht\|_\VV = O(h+\tau) 
\qquad \text{and} \qquad 
\|u - \uht\|_{L^2(Q)} = O(h^2+\tau^2).
\end{align*}
In view of these estimates, we refine space and time simultaneously with
$h\simeq \tau=N^{-1}$.
The discretization errors observed in our numerical tests are reported in
Table~\ref{tab:numerics-balanced}. 
\begin{table}[htp]
\centering
\caption{Convergence for smooth solution on uniform meshes with $h \simeq \tau=N^{-1}$.}
\label{tab:numerics-balanced}
\small
\begin{tabular}{c|c||c|c||c|c}
\toprule
$N$ & $h\simeq\tau$ & $\|u-\uht\|_\VV$ & rate & $\|u-\uht\|_{L^2(Q)}$ & rate \\
\midrule
4   & $2.5000\cdot10^{-1}$ & $7.7348\cdot10^{-1}$ & --    & $3.3345\cdot10^{-2}$ & -- \\
8   & $1.2500\cdot10^{-1}$ & $3.9501\cdot10^{-1}$ & 0.969 & $8.8117\cdot10^{-3}$ & 1.920 \\
16  & $6.2500\cdot10^{-2}$ & $1.9828\cdot10^{-1}$ & 0.994 & $2.2362\cdot10^{-3}$ & 1.978 \\
32  & $3.1250\cdot10^{-2}$ & $9.9151\cdot10^{-2}$ & 1.000 & $5.6134\cdot10^{-4}$ & 1.994 \\
64  & $1.5625\cdot10^{-2}$ & $4.9548\cdot10^{-2}$ & 1.001 & $1.4051\cdot10^{-4}$ & 1.998 \\
128 & $7.8125\cdot10^{-3}$ & $2.4761\cdot10^{-2}$ & 1.001 & $3.5142\cdot10^{-5}$ & 1.999 \\
\bottomrule
\end{tabular}
\end{table}
The results are in perfect agreement with the theoretical predictions. 
In Figure~\ref{fig:numerics-balanced}, we further investigate the behavior of different error contributions. 
%
\begin{figure}[htp]
\centering
\includegraphics[width=0.45\textwidth]{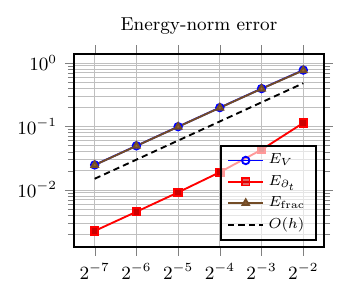} 
\hspace*{0.5cm}
\includegraphics[width=0.45\textwidth]{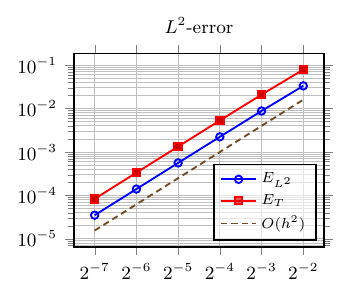} \\[-0.4cm]
\caption{Convergence for the smooth solution under uniform refinement in
space and time. Left: mesh size $h\simeq\tau$ versus the energy error
$E_V=\|u-\uht\|_\VV$ and its components
$E_{\partial_t}=\|\partial_t(u-\uht)\|_{L^2(Q)}$ and
$E_{\mathrm{frac}}=\|\DOP^{1-\alpha/2}\nabla(u-\uht)\|_{L^2(Q)}$.
Right: mesh size $h\simeq\tau$ versus the space--time error
$E_{L^2(Q)}=\|u-\uht\|_{L^2(Q)}$ and the final-time error
$E_T=\|u(T)-\uht(T)\|_{L^2(\Omega)}$.}
\label{fig:numerics-balanced}
\end{figure}
The left plot shows that the fractional-gradient component dominates the
energy error in this example. For our choice $h \simeq \tau$, both components, however, converge with first order, consistent with the interpolation
error estimates discussed in Remark~\ref{rem:energy}.
As predicted by Theorem~\ref{thm:err2}, the space--time $L^2$-error converges with second order. The final-time error converges at the same rate, which is possible from an interpolation point of view, but not fully covered by the arguments of
Remark~\ref{rem:duality}.

\subsection{Reduced temporal regularity}
\label{subsec:numerics-startup}

We next investigate the effect of reduced temporal regularity of the solution $u$ in \eqref{eq:numerical-exact-solution}  for non-integer $1/2<\beta < 3/2$, caused by a mild singularity around $t=0$. 
The relevant regularity for our analysis can here be stated as
\begin{align*}
u \in H^{\beta+1/2}(0,T;H^2(\Omega)).
\end{align*}
The results of Theorem~\ref{thm:err1} and \ref{thm:err2} then immediately lead to the convergence rates
\begin{align*}
\|u - \uht\|_\VV = O(\tau^{\beta-1/2} + h)
\qquad \text{and} \qquad 
\|u - \uht\|_{L^2(Q)} = O(\tau^{\beta+1/2} + h^2).
\end{align*}
These rates are the optimal ones that can be expected in view of the corresponding interpolation errors. 
For uniform meshes with $h \simeq \tau$, we thus expect reduced convergence rates.
In our numerical tests, we consider the two choices $\beta \in \{0.75,1.25\}$. 
Figure~\ref{fig:numerics-uniform-singular} illustrates the convergence behavior observed in our computations. 

\begin{figure}[htp]
\centering
\includegraphics[width=0.45\textwidth]{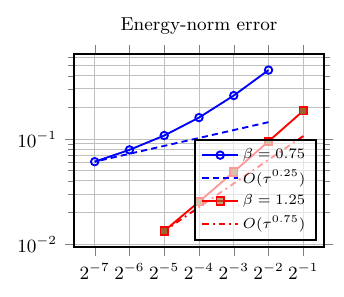} 
\hspace*{0.5cm}
\includegraphics[width=0.45\textwidth]{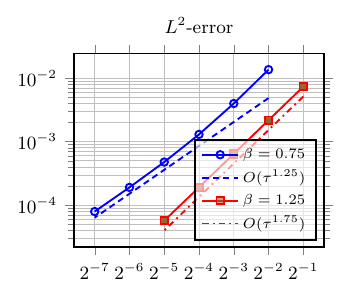} \\[-0.4cm]
\caption{Reduced convergence rates for non-smooth solutions \eqref{eq:numerical-exact-solution} with $\beta\in\{0.75,1.25\}$ and balanced grid sizes $h \simeq \tau$. 
In our computations, we choose $h=\tau/2$ for $\beta=0.75$ and $h=\tau/16$ for $\beta=1.25$.
Left plot: Time step $\tau$ vs. energy-norm error $\|u - \uht\|_\VV$. Right plot: Time step $\tau$ vs. $L^2$-norm error $\|u - \uht\|_{L^2(Q)}$. 
}
\label{fig:numerics-uniform-singular}
\end{figure}
Observe that the energy-norm and $L^2$-errors here contain contributions of different order, leading to a dedicated pre-asymptotic phase. The asymptotic rates obtained for $h \simeq \tau \to 0$, however, are again in very good agreement with the theoretical predictions.  

\begin{remark}
In order to balance the spatial and temporal approximation errors, one could choose the discretization parameters according to $h \simeq \tau^{\beta-1/2}$. This would lead to optimal convergence rates $\|u-\uht\|_\VV=O(h)$ and $\|u-\uht\|_{L^2(Q)}=O(h^2)$, however, at a sub-optimal number of degrees of freedom. 
\end{remark}

\subsection{Graded temporal meshes}
\label{subsec:numerics-graded}

The reduced convergence rates observed in the previous section are caused by the initial layer of the solution $u$ in \eqref{eq:numerical-exact-solution} for non-integer $1/2 < \beta < 3/2$.
This suggests to use local refinement of the time discretization near $t=0$. In order to resolve the initial layer, we therefore consider graded meshes of the form
\begin{align} \label{eq:graded}
    t_n=T\left(n/N_t\right)^\gamma,
    \qquad n=0,\ldots,N_t,
\end{align}
with some $\gamma>1$; see \cite{StynesORiordanGracia2017} for a detailed discussion. For this choice of temporal grid, the largest time step is attained away from the initial layer and satisfies
\begin{align*}
    \tau=\max_{1\le n\le N_t}(t_n-t_{n-1})\simeq N_t^{-1},
\end{align*}
while the step-size close to the singularity is $\tau_{\min} \simeq \tau^{\gamma}$. 
Hence, grading provides additional resolution near $t=0$ without changing
the asymptotic relation between the number of time steps and the maximal
stepsize.
In view of our previous considerations, we choose $\gamma=\frac{2}{2\beta-1}$, which leads to
$\tau_{\min}^{\beta-1/2} \simeq \tau^{\gamma (\beta-1/2)} = \tau$, and allows to establish optimal rates 
\begin{align*}
\|u-\tuht\|_\VV \simeq O(h+\tau) 
\qquad \text{and} \qquad 
\|u - \tuht\|_{L^2(Q)} \simeq O(h^2+\tau^2)
\end{align*}
for the interpolation errors, up to logarithmic powers in $\tau$; 
this can be verified by similar arguments as employed in the proof of Lemma~\ref{lem:approx}. 
As a consequence of the quasi-optimality of the space-time Galerkin approximation, and noting that the duality argument is not affected by the reduced regularity of the solution $u$, we obtain the same optimal convergence rates, up to logarithmic power in $\tau$, also for the finite-element errors.
\begin{figure}[htp]
\centering
\includegraphics[width=0.45\textwidth]{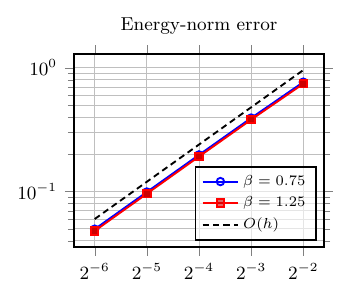} 
\hspace*{0.5cm}
\includegraphics[width=0.45\textwidth]{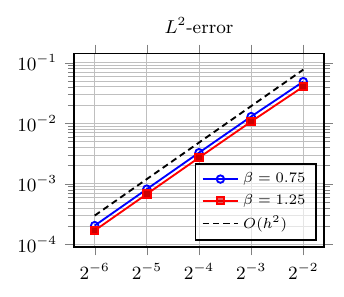} \\[-0.4cm]
\caption{Convergence rates for non-smooth solution \eqref{eq:numerical-exact-solution} on graded meshes \eqref{eq:graded}. For our computations, we choose 
$\gamma=4$ for $\beta=0.75$ and $\gamma=4/3$ for $\beta=1.25$, and we use $n_x=N_t=N$.
Left plot: Meshsize $h=1/N$ vs. energy-norm error $\|u - \uht\|_\VV$.  Right plot:
Meshsize $h=1/N$ vs. $L^2$-error $\|u - \uht\|_{L^2(Q)}$.}
\label{fig:numerics-graded}
\end{figure}

As expected from our previous considerations, the appropriate grading of the temporal mesh allows to regain the full convergence orders without increasing the number of time-steps. 
The sub-optimality by logarithmic factors in $\tau$ is not visible in our computational results.
%

\subsection{Efficient solvers}
\label{subsec:numerics-solver}

As a final step in this section, we now report in more detail on the computational performance of the algorithms used in our numerical tests. Following the remarks of Section~\ref{sec:fast}, we employed two different methods:
\begin{itemize}
\item[(M1)] Block-forward substitution with basic history updates: One spatial finite element problem is solved in every time-step. This involves factorization of the corresponding system matrix and forward/backward substitution. The history residual is computed by summation over all previous time-steps. This method is applicable to non-uniform time-grids and has a complexity of $O(N_x^{3/2} N_t + N_x N_t^2)$. 
\item[(M2)] Fast implementation using block-Toeplitz structure: For uniform time step, the system matrix arising in every time step of the block-forward substitution stays the same and, hence, has to be factorized only once. In every time-step, the spatial finite element systems are solved by forward/backward substitution. The residual history is updated using the recursive fast-convolution approach of \cite{HairerLubichSchlichte1985}. The overall complexity of this algorithm is $O(N_x^{3/2} + N_t N_x \log N_x + N_x N_t \log^2 N_t)$.
\end{itemize}
To obtain a fair comparison of the two methods, we compare their performance on the same test problem mentioned in Section~\ref{subsec:numerics-balanced}.
We fix $n_x=64$, so that $N_x=(64-1)^2=3969$, and use a uniform grid in time. 
In Method (M2), recursive FFT updates are employed for blocks of size $ \ge 128$ only. 
Table~\ref{tab:numerics-solver} reports the total solver
times and the resulting speedup when employing the Toeplitz structure and fast convolutions for the history updates. 

\begin{table}[htp]
\centering
\caption{Comparison of computation times for the two methods (M1) and (M2) employed in our numerical tests. The columns $M_1$ and $M_2$ report the total times for the two algorithms (median over $5$ runs).}
\label{tab:numerics-solver}
\small
%
\setlength{\tabcolsep}{8pt}
\begin{tabular}{c||c|c||c|c||c}
\toprule
$N_t$ & $M_1$ & rate & $M_2$ & rate & speedup\\
\midrule
128  & 0.913  & ---   & 0.038 & ---   & 23.81\\
256  & 1.906  & 1.062 & 0.078 & 1.019 & 24.53\\
512  & 4.178  & 1.133 & 0.166 & 1.092 & 25.22\\
1024 & 9.879  & 1.242 & 0.365 & 1.139 & 27.08\\
2048 & 25.524 & 1.369 & 0.813 & 1.156 & 31.41\\
\bottomrule
\end{tabular}
\end{table}
For medium values of $N_t$, the linear solver complexity $O(N_t N_x^{3/2})$ dominates the overall computational cost of method (M1), which results in an almost linear increase of the complexity with the number of time points on coarse grids.   
For large $N_t$, the complexity of the history computation $O(N_x N_t^2)$ gains significance, which explains the increase of the observed rate for finer time grids. 
In Figure~\ref{fig:numerics-solver}, we display the computation times of the two solution methods and further illustrate those of the two algorithms for the history computation. 

\begin{figure}[htp]
\centering
\includegraphics[width=0.45\textwidth]{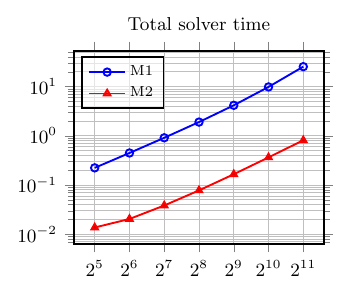} 
\hspace*{0.5cm}
\includegraphics[width=0.45\textwidth]{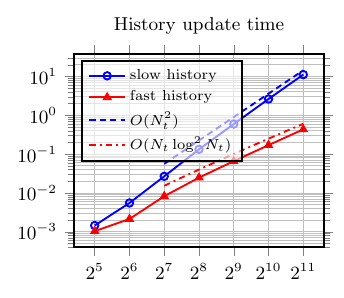} \\[-0.4cm]
\caption{Comparison of two numerical solution strategies: Left plot: $N_t$ vs. total computation times in seconds of method (M1) and (M2). Right plot: $N_t$ vs. time for history updates in seconds.}
\label{fig:numerics-solver}
\end{figure}
The plots clearly illustrate the different complexity of the two different history update strategies. They also indicate that the linear solver times are still dominating for method (M1) for $N_t \le 2048$, resulting in an almost linear increase of the total computations times. 
As predicted by our theoretical considerations, the exploitation of the Toeplitz structure leads to a substantial increase in computational performance.

\section{Discussion}
\label{sec:final}

%
The main contribution of this work is a space-time variational formulation for fractional diffusion based on a natural coercivity structure of the fractional temporal operator. This structure yields well-posedness of the continuous problem and quasi-optimality of the corresponding Galerkin approximations. Based on this framework, we developed a systematic tensor-product finite element discretization and established optimal convergence rates under regularity assumptions that account for the reduced temporal regularity of fractional diffusion problems. The theoretical results are confirmed by the numerical experiments.
Despite its space-time variational formulation, the tensor-product structure permits an efficient realization as a time-stepping procedure. This combines the analytical advantages of the space-time framework with the computational structure of conventional time-stepping methods.

Several extensions of the results presented in this work might be of interest. The regularity theory could be generalized to incompatible initial data, which may lead to stronger singularities near the initial time. Higher-order and adaptive discretizations, including graded temporal meshes, appear to be natural extensions. An efficient realization on nonuniform time grids and the development of a posteriori error estimates would provide further opportunities for adaptive space-time discretization. These directions could also facilitate extensions of the present framework to more general fractional evolution problems.

\subsection*{Acknowledgements}
The work of the first author is supported by the Collaborative Research Centre CREATOR (DFG: Project-ID 492661287/TRR 361; FWF: 10.55776/F90).

\subsection*{Declaration of AI use} 
OpenAI's ChatGPT (GPT-5.6 Luna) was used in preparation of the manuscript for proof-checking, language revision, and typesetting assistance. The authors independently verified the mathematical content and take full responsibility for the manuscript.


\end{document}